\documentclass[12pt,twoside]{article}
\usepackage[inoutnumbered,linesnumbered,algoruled,slide,vlined]{algorithm2e}
\usepackage{sibarticle}
\usepackage{enumerate}
\usepackage{subfigure}
\usepackage{lscape}
\usepackage[latin1]{inputenc}

\newcommand{\RR}{\mathbb{R}}

\newcommand{\tr}{^{\intercal}}

\newcommand{\hspp}{\hspace{8mm}}

\newcommand{\minimize}{\mbox{minimize\hspace{4mm} }}
\newcommand{\subto}{\mbox{subject to\hspace{4mm}}}

\newcounter{commentcounter}
\long\def\symbolfootnote[#1]#2{\begingroup%
\def\thefootnote{\fnsymbol{footnote}}\footnote[#1]{#2}\endgroup}

\newcommand{\lemmanum}[2]{\vspace{3mm} \noindent {\sc Lemma #1}{\it
    #2} \vspace{3mm}}
\newcommand{\half}{\mbox{\textonehalf}}

\title{An Alternative Globalization Strategy for Unconstrained Optimization}
\ShortTitle{Alterntive Globalization Strategy} 
\ShortAuthors{\"{O}ztoprak and Birbil} 
\NumberOfAuthors{2} 
\FirstAuthor{Figen \"{O}ztoprak}
\FirstAuthorAddress{Department of Industrial Engineering, Bilgi
  University, 34060 Istanbul, Turkey.}
\FirstAuthorEmail{figen.topkaya@bilgi.edu.tr}
\FirstAuthorURL{http://engr.bilgi.edu.tr/about-us/academic-staff/figen-topkaya/}
\SecondAuthor{\c{S}. \.{I}lker Birbil} 
\SecondAuthorAddress{Faculty of Engineering and Natural Sciences,
  Sabanc{\i} University, 34956 Istanbul, Turkey.}
\SecondAuthorEmail{sibirbil@sabanciuniv.edu}
\SecondAuthorURL{http://people.sabanciuniv.edu/sibirbil}

\keywords{globalization strategy; unconstrained optimization; parallel
  implementation}

\begin{document}

\maketitle
\begin{abstract}
  We propose a new globalization strategy that can be used in
  unconstrained optimization algorithms to support rapid convergence
  from remote starting points. Our approach is based on using multiple
  points at each iteration to build a representative model of the
  objective function. Using the new information gathered from those
  multiple points, a local step is gradually improved by updating its
  direction as well as its length. We give a global convergence result
  and also provide parallel implementation details accompanied with a
  numerical study. Our numerical study shows that the proposed
  algorithm is a promising alternative as a globalization strategy.
\end{abstract}

\section{Introduction.} 
\label{sec:intro}

In unconstrained optimization, frequently used algorithms, like
quasi-Newton or trust-region methods, need to involve mechanisms that
ensure convergence to local solutions from remote starting
points. Roughly speaking, these \emph{globalization strategies}
guarantee that the improvement obtained by the algorithm is comparable
to the improvement obtained with a gradient step
\cite{Nocedal:2006}. 

In this paper we present a new globalization strategy that can be used
in unconstrained optimization methods for solving problems of the form
\begin{equation}
\label{eq:origprob}
 \min_{x\in\Re^n} \ \ f(x),
\end{equation}
where $f$ is a first order differentiable function. Conventional
methods for solving this problem mostly use local approximations that
are very powerful once the algorithm arrives at the close proximity of
a stationary point. The main idea behind the proposed strategy is
based on using additional information collected from multiple points
to construct an adaptive approximation of the function $f$ in
\eqref{eq:origprob}. In particular, we update our approximate model
constructed around the current iterate and a sequence of trial
points. Our objective is to come up with a better local model than the
one obtained by the current iterate only. We observe that the
collection of this additional information has a profound effect on the
performance of the method as well. Before moving to the next iterate,
the step is improved by updating its direction as well as its length
simultaneously. This step computation involves only the inner products
of vectors. Therefore, each iteration of the algorithm is amenable to
a parallel implementation. Though acquiring the additional information
from multiple points may add an extra burden on the algorithm, this
burden can be alleviated by using the readily available parallel
processors. Furthermore, our numerical experiments demonstrate that
the additional computations at each step may reduce the total number
of iterations, since we learn more about the function structure.

Parallel execution of linear algebra operations is common in the
parallel optimization literature; generally in designing parallel
implementations of existing methods.  Among the earliest work is
\cite{Byrd:1988}, where parallelization of linear algebra steps and
function evaluations in the BFGS method is discussed.  Around the same
time, \cite{Nash:1989} propose to use truncated-Newton methods
combined with computation of the search direction via block iterative
methods. In particular, they focus on the block conjugate-direction
method. By using block iterative methods, they parallelize some steps
of the algorithm, in particular, the linear algebra operations.  For
extensive reviews of the topic, we refer to \cite{Zenios:1989,
  Migdalas:2003, Dennis:2003}.

The use of multiple trial points or directions is also a recurring
idea in the field of parallel nonlinear optimization.  However, unlike
our case, these mostly depend on concurrent information collection /
generation.  \cite{Patel:1984} proposes an algorithm, which is based
on evaluating multiple points in parallel to determine the next
iterate. In the parallel line search implementation of
\cite{Phua:1998}, multiple step sizes are tried in parallel.
\cite{Laarhoven:1985} discusses three parallel quasi-Newton algorithms
that are also based on considering multiple directions to extend the
rank-one updates. To solve the original problem by solving a series of
independent subproblems, \cite{Han:1986} suggests using conjugate
subspaces for quasi-Newton updates. Thus, the resulting method can be
applied in a parallel computing environment.

The rest of this paper is organized as follows.  In Section
\ref{sec:algo}, we discuss the derivation and the theoretical
properties of the proposed strategy.  In Section \ref{sec:perf}, the
success of the main idea of the method is numerically tested, and its
parallel implementation is discussed in Section \ref{sec:imp}.
Conlusion of the study is presented in Section \ref{sec:conc}.

\section{Proposed Globalization Strategy.} 
\label{sec:algo}

At the core of the proposed strategy lies an \emph{extended} model
function, which is constructed by using two linear models at two
reference points. The first of these points is the current iterate and
the second one is a trial point. When one of these trials leads to a
successful step, then the algorithm moves to a new iterate. To make
our discussion more concrete, we start by describing an iteration of
the algorithm.

Let $x_k$ and $s_k$ denote the current iterate and the $k^{th}$ step
of the algorithm, respectively. Then, the next iterate becomes
$x_{k+1} = x_k + s_k$. These are the outer iterations. The vector
$s_k$ is obtained by executing a series of inner iterations and
computing trial steps $s_k^t$, $t=0, 1, \cdots$. We can further
simplify our exposition by defining
\[
\begin{array}{lll}
  f_k:=f(x_k), & g_k:=\nabla f(x_k), & x_k^t := x_k + s^t_k, \\[2mm]
  f^t_k := f(x_k^t), &  g^t_k := \nabla f(x_k^t), & y^t_k := g^t_k - g_k.
\end{array}
\]
The trial step $s_k^t$ is accepted as new $s_k$, if it provides
sufficient decrease in the sense
\begin{equation}
\label{eq:accunc2}
f_k^t - f_k \leq \rho g_k\tr s_k^t \mbox{ for some } \rho\in(0,1).
\end{equation}
This is nothing but the well-known Armijo condition
\cite{Nocedal:2006}.  

Up to this point, the algorithm behaves like a typical unconstrained
minimization procedure. However, to come up with the proposed
globalization strategy, we design a new subprocedure for computing
the trial steps, $s_k^t$. The inner step computation is done by
solving a subproblem that consists of an extended model function
updated at every inner iteration and a constraint restricting the
length of the steps. The extended model function is an approximation
to the objective function, $f$. It is constructed by using the
information gathered around the region bordered by $x_k$ and $x^t_k$.
Actually, the extended model function is a combination of the linear
models of $f$ at $x_k$ and $x^t_k$. That is,
\[
m_k^{t}(s) := \alpha_k^0(s)l_k^0(s) + \alpha_k^t(s)l_k^t(s-s^t_k),
\]
where
\begin{equation}
  l_k^0(s) := f_k + g_k\tr s ~ \mbox{ and } ~ l_k^t(s-s^t_k) := f^t_k + (g^t_k)\tr (s-s^t_k).  
\label{linfunc}
\end{equation}
This construction is illustrated in Figure~\ref{fig:derivation}.  Note
that both weights, $\alpha_k^0(s)$ and $\alpha_k^t(s)$ are functions
of $s$.  We want to make sure that if $s$ is closer to $s_k^t$, then
the weight $\alpha_k^t(s)$ of the linear model around $x_k$
increases. Similarly, $\alpha_k^t(0)$ should increase as $(s-s_k^t)$
gets closer to $(-s_k^t)$.  This is achieved by measuring the length
of the projections of $(s-s_k^t)$ and $s$ on $s_k^t$ by
\begin{equation}
\label{weights}
\alpha_k^0(s) = \frac{(s-s^t_k)\tr (-s^t_k)}{(-s^t_k)\tr (-s^t_k)} ~
\mbox{ and } ~ \alpha_k^t(s) = \frac{s\tr s^t_k}{(s^t_k)\tr s^t_k},
\end{equation}
respectively. 

 \begin{figure}
 \centering
 \includegraphics{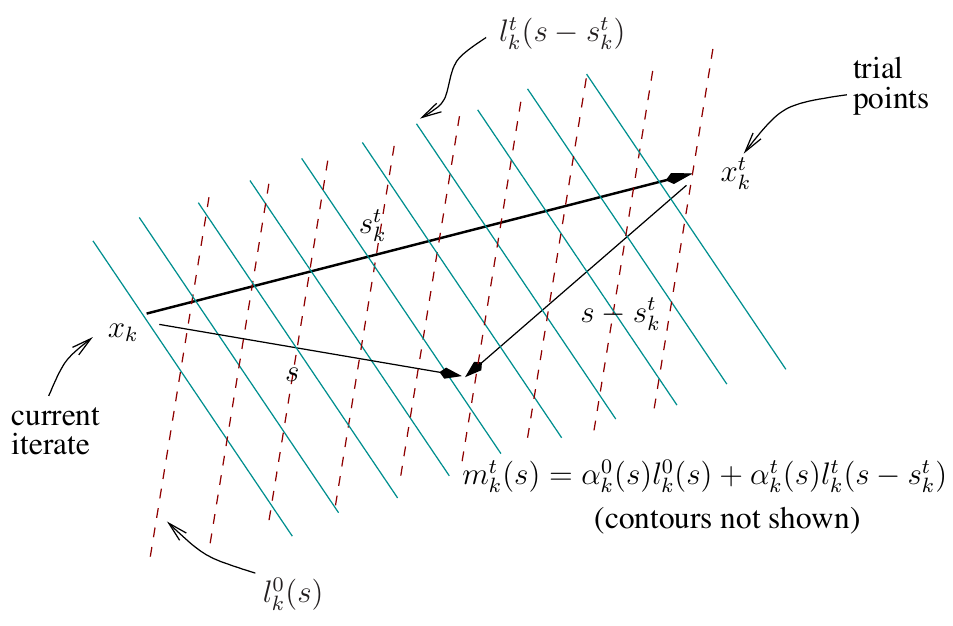}
 \caption{Construction of the model function $m_k^{t}(s)$ around
   $x_k$ and $x_k^t$}
 \label{fig:derivation}
 \end{figure}

Next, we introduce a constraint to handle the possibly hard quadratic
function $m_k^{t}$ as well as control the step length. Recall
that our model function is constructed around the previous trial step
$s_k^t$ by considering the area lying between $x_k$ and $x_k+s_k^t$.
Thus, we impose the constraint
\begin{equation}
\label{const}
 \|s\|^2 + \|s-s^t_k\|^2 \leq \|s^t_k\|^2.
\end{equation}
Clearly, this constraint controls not only the length of $s$ but also
its deviation from the previous trial point $s_k^t$.  In addition,
this choice of the constraint ensures that both weight functions
$\alpha_k^0(s)$ and $\alpha_k^t(s)$ are in $[0,1]$ and they add up to
1. Hence, our model function $m_k^{t}$ is the convex combination of
the linear functions $l_k^0$ and $l_k^t$. This observation is formally
presented in Lemma \ref{lem:tr} and its proof is given in Appendix
\ref{app:proofs}.
\begin{lemma}
\label{lem:tr}
If constraint \eqref{const} holds at inner iteration $t+1$ of
iteration $k$, then both $\alpha_k^0(s)$ and $\alpha_k^t(s)$ are
nonnegative, and they satisfy $\alpha_k^0(s) + \alpha^t(s)=1$.
\end{lemma}

We also make the following observations about the region defined by
constraint \eqref{const}. First, this region is never empty, since any
step $s=\gamma s_k^t$ for $\gamma \in (0,1]$ is in the region. This
implies in a sense that a backtracking operation on the previous trial
step, $s_k^t$ would give a new feasible trial step,
$s_k^{t+1}$. Second, there are always feasible steps in the steepest
descent direction \emph{whenever} $g_k\tr s_k^t \leq 0$ holds. That
is, the steps $s = - \xi g_k$ are feasible for $0 \leq \xi
\leq-\frac{g_k\tr s_k^t}{\|g_k\|^2}$. Thus, as long as the first trial step,
$s_k^0$ is gradient related, each feasible region constructed around a
subsequent trial step involves a feasible point that is also gradient
related.

Algorithm \ref{alg:outline} gives the outline of the proposed
method. The stopping condition is the usual check whether the current
iterate, $x_k$ is a stationary point. We start the inner iterations by
computing the first trial step, $s_k^0$. We shall elaborate on how to
select this initial trial step in Section \ref{sec:imp}. As long as
the trial steps $s_k^t$, $t=0,1,2,\ldots$ are not acceptable, we carry
on with constructing and minimizing the model functions. Once an
acceptable step is obtained by the inner iteration, we evaluate the
next iterate $x_{k+1}$ and continue with solving the overall problem.

\begin{algorithm}
    \label{alg:outline}
    \SetArgSty{textit}
    \KwIn{$x_0$; $\rho \in (0, 1)$; $k=0$}
    \While{$x_k$ is not a stationary point}{
      $t=0$\;
      Compute the first trial step $s_k^0$\;
    \While{$f_k^t - f_k > \rho g_k\tr s_k^t$}{
      Compute the trial step $s_k^{t+1}$ by solving
      \begin{equation}
        \label{eq:initsp}
      \begin{array}{ll}
        \minimize & m_k^{t}(s)=\alpha_k^0(s)l_k^0(s) + \alpha_k^t(s)l_k^t(s-s^t_k) \\
        \subto & \|s\|^2 + \|s - s_k^t\|^2 \leq \|s_k^t\|^2,
      \end{array}
      \end{equation}
      where $l^0, l^t, \alpha^0$, and $\alpha^t$ are given by
      \eqref{linfunc} and \eqref{weights}, respectively\;
      t = t + 1\;
    }
    $x_{k+1}=x_k+s_k^t$\;
  }
   \caption{Outline}
\end{algorithm}

In Algorithm \ref{alg:outline}, the crucial part is solving subproblem
\eqref{eq:initsp}, which has a quadratic objective function and a
quadratic constraint.  After some derivation, the objective function
of this subproblem simplifies to
\begin{equation*}
 m_k^t(s) = f_k + (\bar{g}^t_k)\tr  s  + s\tr  B^t_k s,
\label{model_t}
\end{equation*}
where 
\[
 \bar{g}^t_k := g_k + \frac{1}{\|s_k^t\|^2}(f_k^t - (g_k^t)\tr s_k^t - f_k)s_k^t
\]
and
\[
 B^t_k := \frac{1}{\|s_k^t\|^2}s_k^t(g_k^t-g_k)\tr.
\]
The Hessian of this quadratic objective function is the rank-one
matrix $B^t_k$. Its only nonzero eigenvalue is $(s_k^t)\tr y_k^t$, and
the eigenvectors corresponding to this eigenvalue are $s_k^t$ and
$-s_k^t$.  Furthermore, its eigenvectors corresponding to the zero
eigenvalues are perpendicular to $y_k^t$. Let us note at this point
that the gradient of the model $m_k^t(s)$ is given by
\begin{equation}
\label{eqn:cvxmodgra}
\nabla m_k^t(s) = g_k + \frac{1}{\|s_k^t\|^2}(f_k^t - (g_k^t)\tr s_k^t -
f_k)s_k^t + \frac{1}{\|s_k^t\|^2}[s_k^t(y_k^t)\tr +
{y_k^t}(s_k^t)\tr] s.
\end{equation}

\subsection{Convex Objective Function.}
\label{sec:initial_design}

Before considering the general case, we discuss in this section how
subproblem \eqref{eq:initsp} can be solved when the objective function
is convex.  We shall soon observe that it is not straightforward to
obtain trial steps that yield sufficient descent with this initial
design. Thus, we will propose modifications over this first attempt in
Section~\ref{sec:final_design} to guarantee convergence.  However, we
still present here our observations with the convex case as it
constitutes the basis of our approach for general functions given in
the next section.

We first check whether the subproblem always provides a nonzero step
at every inner iteration. It turns out that if the most recent trial
step, $s_k^t$ is gradient related, then the next trial step,
$s_k^{t+1}$ cannot be the zero vector. Furthermore, if the objective
function value of the recent trial step is strictly greater than the
current objective function value, then the optimal solution of the
subproblem is in the interior of the region defined by the constraint
\eqref{const}. Lemma \ref{lem:fact1} summarizes this discussion.  The
proof of this lemma is given in Appendix \ref{app:proofs}.
\begin{lemma}
\label{lem:fact1}
Let $f$ be a convex function and consider the inner iteration $t+1$ of
iteration $k$ with $\|s_k^t\| \neq 0$. If $g_k\tr s_k^t \leq 0$, then
$\|s_k^{t+1}\| \neq 0$ and
\[
\alpha^*:=\underset{0 \leq \alpha \leq 1}{\arg\min} \ m_k^t (\alpha s_k^t) = \min\left\{
\frac{(y_k^t)\tr s_k^t - f_k^t + f_k}{2(y_k^t)\tr s_k^t}, 1 \right\}.
\]
If we additionally have $f_k^t > f_k$, then $\alpha^* < 1$.
\end{lemma}
Note that we require condition $g_k\tr s_k^t \leq 0$ hold for all
trial steps to guarantee nonzero trial steps.  So, a convergence
argument for the proposed algorithm is not clear unless it ensures
this condition.  We shall visit this issue later in this section.

We next show how subproblem \eqref{eq:initsp} can be solved, if the
conditions given in Lemma \ref{lem:fact1} hold. This subproblem can be
rewritten as
\[
\begin{array}{ll}
  \minimize & [(f_k^t- (g_k^t)\tr s_k^t - f_k)s_k^t +
    \|s_k^t\|^2g_k]\tr s + s\tr [s_k^t(g_k^t-g_k)\tr]s\\ \subto & s\tr
  (s - s_k^t) \leq 0.
\end{array}
\]
Note that this is a convex program as $(s_k^t)\tr (g_k^t-g_k)\geq 0$
by the convexity assumption on $f$ in this section.  Let $s_k^{t+1}$
denote the optimal solution of the above problem and $\beta \geq 0$ be
the Lagrange multiplier corresponding to the constraint. The
Karush-Kuhn-Tucker optimality conditions imply
\begin{align}
  (f_k^t- (g_k^t)\tr s_k^t - f_k - \beta)s_k^t + \|s_k^t\|^2g_k +
  s_k^t(y_k^t)\tr s_k^{t+1} + y_k^t(s_k^t)\tr s_k^{t+1} + 2\beta s_k^{t+1} & = 0, \label{eqn:kkt1}\\
  \beta (s_k^{t+1} - s_k^t)\tr s_k^{t+1} & = 0. \label{eqn:kkt2}
\end{align}
Simplifying \eqref{eqn:kkt1} leads to
\begin{equation}
\label{eqn:cvxopt}
s_k^{t+1} = c_g(\beta) g_k + c_s(\beta) s_k^t + c_y(\beta) y_k^t,
\end{equation}
where
\[
c_g(\beta) = -\frac{\|s_k^t\|^2}{2\beta}, \qquad c_s(\beta) =
-\frac{\|s_k^t\|^2}{2\beta} \left(\frac{\delta}{\|s_k^t\|^2} -
  \frac{(y_k^t)\tr s_k^t + 2\beta}{\theta}((y_k^t)\tr g_k +
  \frac{\delta}{\|s_k^t\|^2} (y_k^t)\tr s_k^t ) +
  \frac{\|y_k^t\|^2}{\theta} ((s_k^t)\tr g_k + \delta)\right),
\]
\[
c_y(\beta) = -\frac{\|s_k^t\|^2}{2\beta} \left(- \frac{(y_k^t)\tr
    s_k^t + 2\beta}{\theta}((s_k^t)\tr g_k + \delta) +
  \frac{\|s_k^t\|^2}{\theta} ((y_k^t)\tr g_k +
  \frac{\delta}{\|s_k^t\|^2}(y_k^t)\tr s_k^t)\right),
\]
with
\[
 \delta = f_k^t-f_k-(g_k^t)\tr s_k^t-\beta \mbox{ and } \theta = \left((y^t_k)\tr s_k^t +
2\beta\right)^2-\|s_k^t\|^2\|y_k^t\|^.
\]
Consider first the case when $\beta = 0$.  Then, using Lemma
\ref{lem:fact1}, we have $s_k^{t+1}=\alpha^*s_k^t$.  When $\beta > 0$,
the optimal solution is at the boundary of the constraint. If we now
use \eqref{eqn:cvxopt} with \eqref{eqn:kkt2}, then after some
derivation we obtain
\[
 \frac{1}{4\beta\theta^2}P(\beta)=0,
\]
where $P(\beta)$ is a sixth order polynomial function of $\beta$ (see
the online
supplement\footnote{\url{http://people.sabanciuniv.edu/sibirbil/glob_strat/OnlineSupplement.pdf}}
for the details).  Thus, the subproblem can be solved by computing the
roots of this polynomial.

In fact, it is not hard to show that if the inner iteration step
$s_k^t$ satisfies $s_k^t=-M^tg_k$ for some matrix $M^t$, then at the
next inner iteration we have $s_k^{t+1}=-M^{t+1}g_k$ 
  However, it is not possible to guarantee that the
matrix $M^{t+1}$ will be positive definite given $M^t$ is.  That is
because the gradient of the model function is given as the conical
combination of $g_k$ and $-s_k^t$.  First of all, the model gradient
can be zero for $s=0$ at a point $x_k$ with $g_k\neq 0$.  Also, when
the deviation of the objective function $f$ from linearity is
relatively large, then the model gradient can point a \emph{backward}
direction; i.e. the reverse direction of $s_k^t$.  Since the
constraint of the model does not allow such a backward step, this can
cause $s_k^{t+1}$ to be \emph{zero}.  Moreover, the model gradient is
scaled with a matrix, and the scaled step adds a third component along
$y_k$ to the linear combination whose coefficient is not necessarily
positive.  So, the resulting $s_k^{t+1}$ might not satisfy
$g_k^Ts_k^{t+1}\leq 0$ even if $s_k^t$ does.

\subsection{General Objective Function.} 
\label{sec:final_design}

Using our observations on the convex case, we next concentrate on a
construction where the model function $m_k^t$ has the same gradient
value at $s=0$ as the original objective function, $f$. 
Then, we also add a regularization term so that 
the subproblem can be made convex even if $f$ is not.

We start by relaxing the constraint \eqref{const} and moving it to
the objective function. That is
\[
 \underset{s \in \RR^n}{\min} \{m^t_k(s) + \sigma_1 (s-s_k^t)\tr s\}.
\]

We choose
\[
\sigma_1 =  \frac{1}{\|s_k^t\|^2}( f_k^t - (g_k^t)\tr s_k^t - f_k ),
\]
so that $\nabla m_k^t(0) = g_k$. 
Observe that the above choice of $\sigma_1$ can be negative (for $f$ nonconvex), in which case the constraint operates the reverse way;
i.e. it tries to keep $s$ away from $s_k^t$.  Since that causes us to loose our control on the size of $s$, we add a regularization term
and the subproblem becomes
\[
 \underset{s \in \RR^n}{\min}\{m^t_k(s) + \sigma_1 (s-s_k^t)\tr s + \sigma_2 s\tr s\}.
\]
After simplifying the objective function, we obtain our new subproblem
for the general case
\begin{equation}
\label{eqn:newsub}
 \underset{s \in \RR^n}{\min} \{g_k^Ts + s\tr \frac{1}{\|s_k^t\|^2}[\sigma I + s_k^t(y_k^t)\tr] s\},
\end{equation}
where
\[
 \sigma = (\sigma_1 + \sigma_2)\|s_k^t\|^2  =  (f_k^t - (g_k^t)\tr s_k^t - f_k) + \sigma_2\|s_k^t\|^2.
\]

Note that the model function \eqref{eqn:newsub} is always
convex provided that the regularization parameter $\sigma_2$ is chosen
sufficiently large. Beyond the convexity of $m_k^t$, we want to
guarantee for some $\eta \in (0, 1)$ that the steplength condition
\begin{equation}
\label{eq:bkr}
 \|s_k^{t+1}\| \leq \eta \|s_k^t\|
\end{equation}
holds. Suppose $s_k^{t+1}$ is such a step and it is the optimal solution for
problem \eqref{eqn:newsub} with a convex objective function. Then, the
first order optimality condition implies
\begin{equation}
\label{eq:brt}
 s_k^{t+1} = -\|s_k^t\|^2(\bar B_k^t)^{-1}g_k,
\end{equation}
where
\[
\bar B_k^t = 2\sigma I + s_k^t(y_k^t)\tr + y_k^t(s_k^t)\tr.
\]
Thus, we obtain
\begin{align*}
\|s_k^{t+1}\| & \leq \|s_k^t\|^2 \left\|(\bar B_k^t)^{-1}\right\|\|g_k\|\\
        &  = \left( \|s_k^t\|\|g_k\| \lambda_{\min}(\bar B_k^t)^{-1}\right) \|s_k^t\|,
\end{align*}
where $\lambda_{\min}$ denotes the smallest eigenvalue of its matrix
parameter. It is not difficult to see that $2\sigma$ is an eigenvalue
of $\bar B_k^t$ with multiplicity $n-2$, and the
remaining two eigenvalues are the extreme eigenvalues of
$\bar B_k^t$ given by
\[
 \lambda_{\max}(\bar B_k^t) = 2\sigma + \|y_k^t\|\|s_k^t\|+(y_k^t)\tr s_k^t \quad \mbox{ and } \quad \lambda_{\min}(\bar B_k^t) = 2\sigma - \|y_k^t\|\|s_k^t\|+(y_k^t)\tr s_k^t,
\]
where $\lambda_{\max}$ denotes the largest eigenvalue of its matrix
parameter. 

To obtain a convex model, we need
\begin{equation}
\label{eq:ays}
 2\sigma-\|y^t_k\|\|s_k^t\|+(y^t_k)\tr s_k^t > 0. 
\end{equation}

On the other hand, the steplength condition \eqref{eq:bkr} requires
\begin{equation}
\label{eqn:sigma2}
\|s_k^t\|\|g_k\|\lambda_{min}(\bar B_k^t)^{-1}\leq \eta \quad \Rightarrow \quad 2\sigma-\|y^t_k\|\|s_k^t\|+(y^t_k)\tr s_k^t \geq \frac{\|s_k^t\|\|g_k\|}{\eta}
\end{equation}
Since the latter bound is larger, using relation \eqref{eqn:sigma2} for choosing $\sigma_2$ provides the convexity
requirement and satisfies the steplength condition.  That is
\[
 \sigma = \|s_k^t\|^2 (\sigma_2+\sigma_1) \geq \frac{1}{2}\left(\|s_k^t\|\left(\|y_k^t\|+\frac{1}{\eta}\|g_k\|\right)-(y_k^t)\tr s_k^t\right).
\]

Our last step is to state the minimizer $s^{t+1}_k$ of
\eqref{eqn:newsub} with selected $\sigma_1$ and $\sigma_2$
values. Following similar steps as in Section
\ref{sec:initial_design}, after some derivation we obtain
\begin{equation}
\label{step_dkt}
 s^{t+1}_k = c_g(\sigma) g_k + c_y(\sigma) y^t_k + c_s(\sigma) s^t_k, 
\end{equation}
where
\[
 c_g(\sigma) = -\frac{\|s_k^t\|^2}{2\sigma}, \quad c_y(\sigma) =
 -\frac{\|s_k^t\|^2}{2\sigma\theta}[-((y_k^t)\tr s_k^t +
   2\sigma)((s_k^t)\tr g_k) + \|s_k^t\|^2((y_k^t)\tr g_k)],
\]
\[
 c_s(\sigma) = -\frac{\|s_k^t\|^2}{2\sigma\theta}[-((y_k^t)\tr s_k^t + 2\sigma)((y_k^t)\tr g_k) + \|y_k^t\|^2((s_k^t)\tr g_k)],
\]
with 
\[
\theta = \left((y^t_k)\tr s_k^t + 2\sigma\right)^2-\|s_k^t\|^2\|y_k^t\|^2, 
\]
and
\begin{equation}
\label{eqn:sigma}
 \sigma =  \half\left(\|s_k^t\|\left(\|y_k^t\|+\frac{1}{\eta}\|g_k\|\right) - (y_k^t)\tr s_k^t\right).
\end{equation}
It is important to observe that the solution of the subproblem does
not require storing any matrices and the main computational burden is
only due to inner products, which could be done in parallel
(see Section \ref{sec:imp}).

When it comes to the convergence of the proposed algorithm, we are
saved basically by keeping the directions of its steps \emph{gradient
  related}.  However, we could not directly refer to the existing
convergence results, since the directions in our algorithm change
during inner iterations and the steplength is not computed by using a
one dimensional function. In the following convergence note, we only
assume additionally that the objective function $f$ is bounded below
and its gradient $\nabla f$ is Lipschitz continuous with parameter
$L$.

\begin{lemma}
\label{lemm:stepunc2}
Suppose that the first trial step $s_k^0$ satisfies 
$m_0\|g_k\|\leq\|s_k^0\|\leq M_0\|g_k\|$ and $0 \geq s_k^0g_k\geq
-\lambda_0\|g_k\|^2$ for some $m_0,M_0,\lambda_0\in(0,\infty)$. Then,
at any iteration $k$, the optimal solution $s_k^{t+1}$ in
\eqref{step_dkt} becomes an acceptable step satisfying
\eqref{eq:accunc2} in finite number of inner iterations at a
nonstationary point $x_k$.
\end{lemma}

\begin{proof}
  Consider any iteration $k$.  Since $s_k^{t+1}$ is obtained by
  \eqref{eq:brt}, the steps computed at each inner iteration satisfy
  $-g_k\tr s_k^{t+1}\geq \lambda_{t+1}\|g_k\|^2$,
  $t=0, 1, 2, \cdots$ for
  $\lambda_{t+1}=\|s_k^t\|^2(\lambda_{max}(\bar B_k^t))^{-1}$.  Note for $t=0, 1, \cdots$ that $\lambda_{t}>0$ by the
  choice of $\lambda_0$ and the relation \eqref{eq:ays}. We have
\begin{align*}
\lambda_{t+1} \frac{\|g_k\|^2}{\|s_k^{t+1}\|^2} & \geq \lambda_{t+1} \frac{\|g_k\|^2}{\eta^2\|s_k^t\|^2}\\
& = \|s_k^t\|^2\left(\|s_k^t\|\left(\|y_k^t\|+\frac{1}{\eta}\|g_k\|\right) 
- (y_k^t)\tr s_k^t + \|y_k^t\|\|s_k^t\|+(y_k^t)\tr s_k^t\right)^{-1} \frac{\|g_k\|^2}{\eta^2\|s_k^t\|^2}\\ 
& = \frac{\|g_k\|^2}{\|s_k^t\|\left(2\|y_k^t\|+\eta^{-1}\|g_k\|\right)\eta^2} 
\geq \frac{\|g_k\|^2}{\|s_k^t\|\left(2L\|s_k^t\|+\eta^{-1}\|g_k\|\right)\eta^2}.
\end{align*}
Using \eqref{eq:bkr}, we have $\|s_k^t\| \leq \eta_t M_0 \|g_k\|$. Then,
\[
\lambda_{t+1} \frac{\|g_k\|^2}{\|s_k^{t+1}\|^2} \geq \frac{\|g_k\|^2}{\left(2L\eta^{2t}M_0^2+\eta^tM_0\eta^{-1}\right)\|g_k\|^2\eta^2} 
= \frac{1}{2L\eta^{2t+2}M_0^2+\eta^{t+1}M_0}.
\]
Therefore, we obtain
\begin{equation}
\label{eq:chg}
 \lambda_t \frac{\|g_k\|^2}{\|s_k^t\|^2} \geq \frac{1}{2L\eta^{2t+1}M_0^2+\eta^tM_0}.
\end{equation}
Suppose that the acceptance criterion \eqref{eq:accunc2} is never
satisfied as $t\rightarrow\infty$.  Then,
\[ 
f_k^t - f_k > \rho g_k\tr s_k^t , \quad \mbox{ for all } t.
\] 
This implies
\[
 g_k\tr s_k^t + \frac{L}{2}\|s_k^t\|^2 > \rho.
\]
Next the desired inequality is obtained by
\[
    \frac{L}{2} > (1-\rho)
    \frac{(-g_k\tr s_k^t)}{\|s_k^t\|^2} \geq (1-\rho) \lambda_t
    \frac{\|g_k\|^2}{\|s_k^t\|^2} \geq (1-\rho) \frac{1}{2L\eta^{2t+1}M_0^2+\eta^tM_0}.
\]
Since $\eta\in(0,1)$, the right hand side of the last inequality
increases without a bound as $t\rightarrow \infty$.  This gives a
contradiction as $L$ is bounded. So, an acceptable $\|s_k^t\|$ should
be obtained in finite number of inner iterations.
\end{proof}

Note that the conditions on the initial step length, $s_k^0$ in Lemma
\ref{lemm:stepunc2} are simply satisfied, if we choose $s_k^0 =
-\epsilon g_k$ for any $ \epsilon \in (0, 1]$. This implies
  $m_0=M_0=\lambda_0=\epsilon$.

\begin{theorem}                                                                                          
Let $\{x_k\}$ be the sequence of iterates generated
  by Algorithm \ref{alg:parimp} and $\{s_k^0\}$ satisfy the
  requirements in Lemma~\ref{lemm:stepunc2}.  Then, any limit point of
  $\{x_k\}$ is a stationary point of the objective function $f$.
\end{theorem}

\begin{proof}
By Lemma~\ref{lemm:stepunc2}, for any $k$, \eqref{eq:accunc2} is
satisfied after a finite number of inner iterations, say $t(k)$.  Also,
note that relation \eqref{eq:brt} implies $\|s_k^t\|\geq \lambda_t
\|g_k\|$ with the same $\lambda_t$ as in \eqref{eq:chg}.  We will now
derive a lower bound for $\lambda_{t(k)}$:
\begin{align*}
\lambda_{t+1} & =
\|s_k^t\|^2\left(\|s_k^t\|\left(\|y_k^t\|+\frac{1}{\eta}\|g_k\|\right)
- (y_k^t)\tr s_k^t + \|y_k^t\|\|s_k^t\|+(y_k^t)\tr
s_k^t\right)^{-1}\\ 
& = \frac{\eta\|s_k^t\|}{\|g_k\|+2\eta\|y_k^t\|} \geq
\frac{\eta\|s_k^t\|}{\|g_k\|+2L\eta\|s_k^t\|} \geq
\frac{\eta}{\|g_k\|+2L\eta^{t+1}M_0\|g_k\|}\|s_k^t\|.
\end{align*}
So, 
\[
 \lambda_1 \geq \frac{\eta}{1+2L\eta M_0}m_0 \quad \mbox{and} \quad \lambda_t\geq \frac{\eta}{1+2L\eta^t M_0}\lambda_{t-1}
 \geq \frac{\eta}{1+2L\eta M_0}\lambda_{t-1} \geq \left(\frac{\eta}{1+2L\eta M_0}\right)^t m_0, \ t=2,3,\cdots
\]
Recall that for the accepted step $s_k$ of iteration $k$ we have
$-s_k\tr g_k\geq\lambda_{t(k)}\|g_k\|^2$.  Clearly, $\lambda_{t(k)}>0$
is bounded away from zero as $t(k)$ is finite. Consider now any
subsequence of $\{x_k\}$ with indices $k\in \mathcal{K}$ such that
\[
 \underset{k\in\mathcal{K}}{\lim} \: x_k = \hat{x}.
\]
Then, for any
$k,k^\prime\in \mathcal{K}$ with $k^\prime > k$ we have
\begin{equation}
\label{unc2son}
f_k - f_{k^\prime} \geq f_k - f_k^t \geq -\rho g_k\tr s_k \geq \rho\lambda_{t(k)} \|g_k\|^2.
\end{equation}
Since $\{x_k\}_{k\in\mathcal{K}}$ converges to $\hat{x}$, the
continuity of $f$ implies that $f_k\rightarrow \hat{f}$ as
$k\in\mathcal{K}$. Therefore $f_k - f_{k^\prime}\rightarrow 0$ as
$k,k^\prime\rightarrow\infty$. Thus, we obtain $\nabla f(\hat{x}) = 0$
by \eqref{unc2son} since $\lambda_{t(k)}$ is bounded away from zero.
\end{proof}

\subsection{Relationship with Quasi-Newton Updates. }
\label{sec:quasi}
Observing the formula in \eqref{eq:brt}, we next question the
relationship between the step computation of the proposed algorithm
and the quasi-Newton updates, in particular the DFP formula
\cite{Nocedal:2006}. The key to understanding this relationship is to
treat the current iterate $x_k$ as the previous iterate, whereas our
trial step $x_k^t$ becomes the next iterate in standard quasi-Newton
updates.

Let $D_{k}$ denote the quasi-Newton approximation to the Hessian at
iterate $x_k$. Then, the infamous DFP formula yields the new
approximation to the Hessian in the next iterate $x_{k+1}$ as
\[
D_{k+1} = \left(I - \frac{y_ks_k\tr}{y_k\tr s_k}\right)D_k\left(I -
  \frac{s_ky_k\tr}{y_k\tr s_k}\right) + \frac{y_ky_k\tr}{y_k\tr s_k},
\]
where $s_k = x_{k+1} - x_k$ and
$y_k = \nabla f(x_{k+1}) - \nabla f(x_{k})$. We first substitute
$s_k^t$ for $s_k$ and $y_k^t$ for $y_k$ in the above equation. Then,
by setting
\[
D_k = -\frac{(y_k^t)s_k^t}{\|s_k^t\|^2} I,
\]
we obtain
\[
D_{k+1} = \frac{1}{\|s_k^t\|^2} \left((-y_k^t)\tr s_k^t I + s_k^t (y_k^t)\tr +
y_k^t (s_k^t)\tr\right).
\]
Let us now denote the Hessian of the objective our model function
given in relation \eqref{eqn:newsub} by $H_k^{t}$. Then, we have
\[
H_k^{t} = \frac{1}{\|s_k^t\|^2} \left(2 \sigma I + s_k^t (y_k^t)\tr +
y_k^t (s_k^t)\tr\right).
\]
Next we plug in the value of $\sigma$ from \eqref{eqn:sigma} and
obtain 
\[
H_k^{t} = \frac{1}{\|s_k^t\|^2} \left((-y_k^t)\tr s_k^t I + s_k^t (y_k^t)\tr +
y_k^t (s_k^t)\tr\right) + \left(\frac{\|y_k^t\|+\eta^{-1}\|g_k\|}{\|s_k^t\|}\right) I.
\]
The last term above, in a sense, serves the purpose of obtaining a
positive definite matrix. Thus, the objective function becomes
convex. Looking at our approach from the quasi-Newton approximation
point of view, we again confirm that the additional information
collected from the trial iterate $x_k^t$ is used to come up with a
better model around the current iterate $x_k$. Our derivation above
also shows that the construction of the model function $m^{t+1}(s)$
from one inner iteration to the next is completely
independent. 

\section{Practical Performance.}
\label{sec:perf}

In this section, we shall conduct a numerical study to demonstrate the
novel features of the proposed globalization strategy. We have
compiled a set of unconstrained optimization problems from the
well-known CUTEst collection \cite{Cute:2004}. We have also embedded
the proposed strategy (PS) into the L-BFGS package
\cite{lbfgs:package} as a new globalization alternative to the
existing two line search routines; backtracking (BT) and More-Thuente
line search (MT).

First, let us give a two-dimensional example to illustrate how the
proposed strategy can be beneficial. Figure \ref{fig:rose} shows the
contours of the Rosenbrock function. The inner iterates of PS are
denoted by $+$ markers, the inner iterates of BT are denoted by
$\circ$ markers, and the inner iterates of MT are denoted by $\times$
markers. All three globalization strategies start from the same
initial point. The numbers next to the trial points denote the inner
iterations. The corresponding objective function values are given in
the subplot at the southeast corner of the figure.

\begin{figure}
\centering
\includegraphics[trim={3cm 4cm 0 3cm}, scale=0.7]{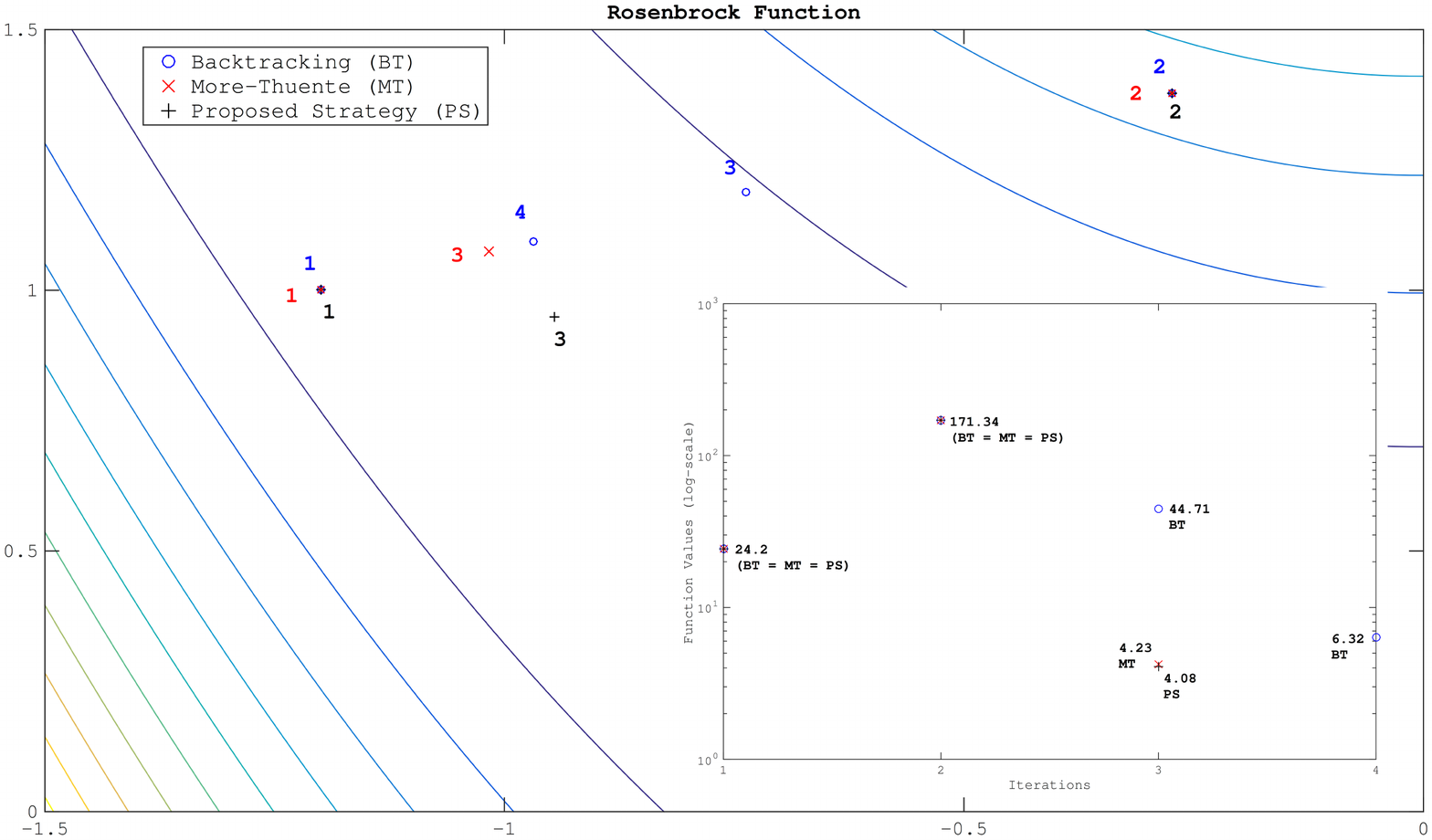}
\caption{The inner iterations obtained with the two line search methods
and the proposed strategy on the Rosenbrock function.}
\label{fig:rose}
\end{figure}

Figure \ref{fig:rose} illustrates the main point of the proposed
strategy: PS can change the direction of the trial step as well as its
length.  It is important to note that PS is not a trust-region method
because it builds a new model function by using the current trial step
for computing the next trial step. However, in a trust-region method
the model function is not updated within the same iteration and the
current trial step does not contribute to the computation of the next
trial step. For this particular two-dimensional problem, PS performs
better than the other two globalization strategies as it attains the
lowest objective function value after three inner iterations.

Next, we test the effect of the direction updates introduced by the
proposed strategy. We solve the CUTEst problems by employing the
backtracking line search with Wolfe conditions using the default
reduction factor of 0.5 given in the L-BFGS package. For the proposed
strategy, we also set the parameter $\eta$ to 0.5. Therefore, the new
strategy requires the same amount of steplength decrease (but it may
additionally alter the search direction). In our tests, we start with
both the L-BFGS step with memory set to 5 ($m=5$) and the gradient
step ($m=0$).  The maximum number of iterations is set to 500 for the
L-BFGS step, and to 1,000 for the gradient step. The comparison is
carried out for those problems, where both methods obtained the same
local solution (which is assessed by comparing the objective function
values and the first two components of the solution vectors).  The
results are reported for 30 problems when $m=0$, and for 63 problems
when $m=5$.

\begin{figure}
\centering
\includegraphics[trim={0 3cm 0 0}, scale=0.30]{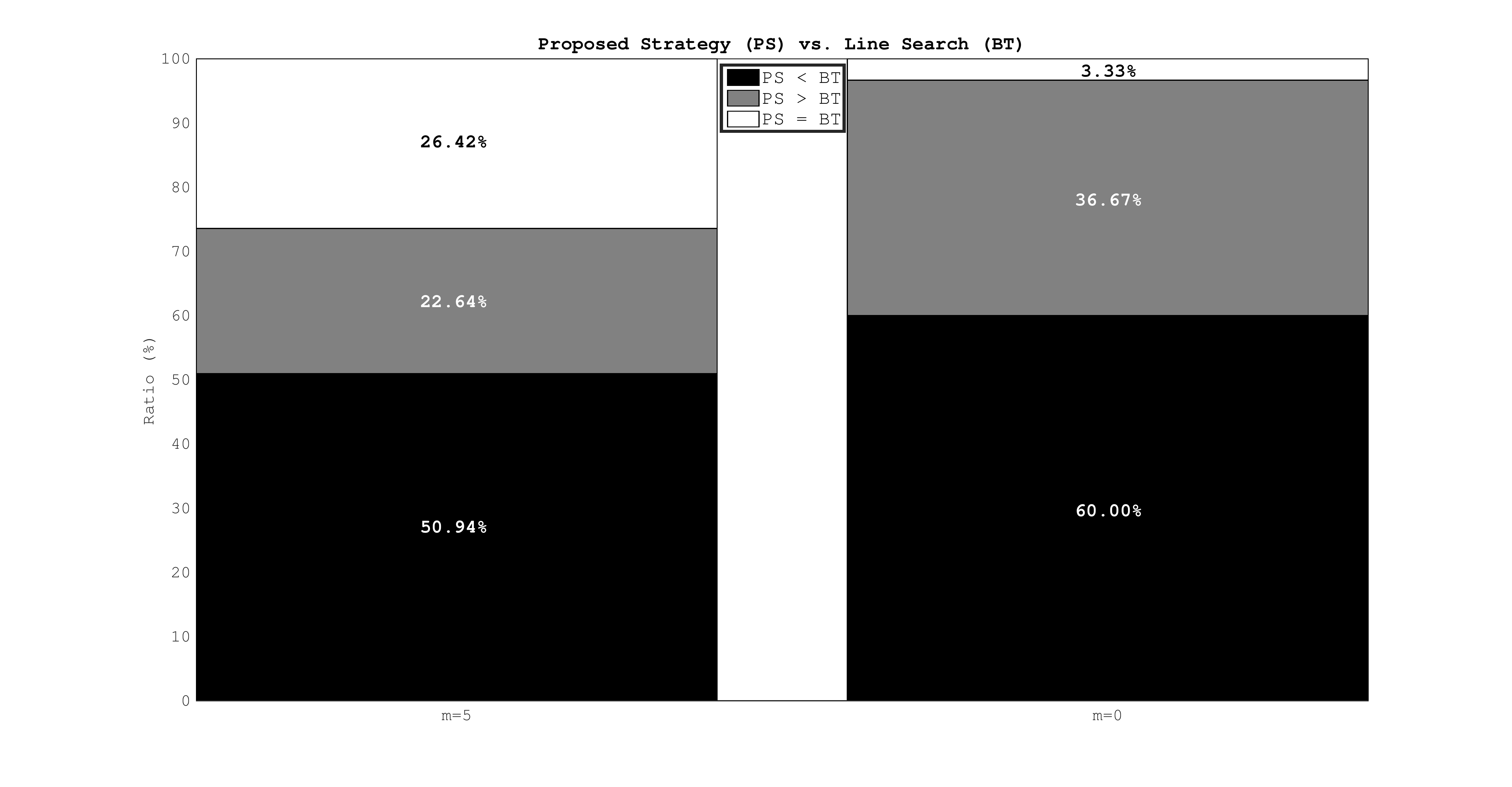}
\caption{Comparison of the proposed strategy against the line search
  method in terms of total number of function evaluations.}
\label{fig:compare}
\end{figure}

Figure~\ref{fig:compare} summarizes our observations. We compare two
strategies in terms of the total number of function evaluations
required to solve the problems. When $m=0$, solving $60\%$ of the
problems required less number of function evaluations than the line
search. The same figure becomes $50.94\%$ when $m=5$, and solving
$26.42\%$ of the problems took exactly the same number of iterations
with both strategies. Overall, these results show that the direction updates 
of the new strategy can indeed provide improvements.

\section{Parallel Implementation.}  
\label{sec:imp}

In this section, we will provide an outline of the subproblem solution
operations while we discuss their parallel execution. Then, we conduct
some numerical experiments to illustrate the parallel performance of the
proposed algorithm. 

As relation \eqref{step_dkt} shows, the basic trial step at inner
iteration $t+1$ is in the subspace spanned by $g_k$, $g_k^t$, and the
previous trial step $s_k^t$.  We should also emphasize that the
computation of $s_k^{t+1}$ requires just a few floating point
operations, once the necessary inner products are completed.  In
particular, the computation of \eqref{step_dkt} requires the following
inner products.
\begin{equation}
\begin{array}{l}
v_1 = (s^t_k)\tr y^t_k, \quad\quad v_2 = (s^t_k)\tr s^t_k, \quad\quad v_3 = (y^t_k)\tr y^t_k, \\[2mm]
v_4 = (y^t_k)\tr g_k, \quad\quad v_5 = g_k\tr g_k, \quad\quad\quad v_6 = (s^t_k)\tr g_k. \\ 
\end{array}
\label{eq:p1}
\end{equation}
Once the scalar values $v_i$, $i=1, \cdots,6$ are available, the
multipliers $c_g(\sigma)$, $c_s(\sigma)$, $c_y(\sigma)$ in
\eqref{step_dkt} can be computed in a total of 32 floating point
operations as follows:
\begin{equation}
 \begin{array}{c}
   \sigma =
  \half(\sqrt{v_2}(\sqrt{v_3}+\frac{1}{\eta}\sqrt{v_5})-v_1),
  \quad\quad \theta = (v_1+2\sigma)^2-v_2v_3, \quad \quad c_g(\sigma) = \frac{-v_2}{2\sigma},\\[2mm]
   c_s(\sigma) = c_g(\sigma)\left(-\frac{v_2+2\sigma}{\theta}v_4 + \frac{v_3}{\theta}v_6\right), \quad\quad c_y(\sigma) = c_g(\sigma)\left(-\frac{v_2+2\sigma}{\theta}v_6 + \frac{v_2}{\theta}v_4\right).
 \end{array}
\label{eq:p2}
\end{equation}
Clearly, the total cost of synchronization operations is negligible.
Thus, the parallel execution of the above inner products determine the
parallel performance as well as the two remaining components of the
algorithm: the computation of initial trial step $s_k^0$ and
function/gradient evaluations.  We require that $s_k^0$ is computed
without introducing too much sequential work to the overall algorithm,
and satisfy the conditions in Section~\ref{sec:algo}.  A simple choice
could be the gradient step, which we also preferred in our own
implementation.

The parallel implementation given in Algorithm \ref{alg:parimp} is
considered for a shared memory environment with $p$ physical cores
(threads). 
Note that the parallel implementation of function/gradient computations is clearly problem
dependent. However --at least partial-- separability is likely to occur
for large-scale problems. This issue, we assume, is taken care of by
the user.  


\begin{algorithm}
    \SetArgSty{textit}
  \KwIn{$x_0$; $\rho \in (0, 1)$; $k=0$}
  \While{$x_k$ is not a stationary point}{ $t=0$\; Compute the first
    trial step $s_k^0$ \hspp (Parallel)\; $x_k^0 = x_k + s_k^0$ \hspp (Parallel)\;
    $\Delta = g_k\tr s_k^0$ \hspp (Parallel)\; \For{$t=1,2, \cdots$}{ Compute
      $f_k^t$, $g_k^t$ \hspp (Parallel)\; \If{$f_k - f_k^t \geq -\rho
        \Delta$}{ $x_{k+1} = x_k^t$, $f_{k+1} = f_k^t$, $g_{k+1} =
        g_k^t$\; $k=k+1$\; \textbf{break}\;
	}
	Compute $v_i,i=1,2, \cdots,6$ as in \eqref{eq:p1} \hspp (Parallel, $O(n/p)$)\; 
	Compute $c_g,c_s,c_y$ as in \eqref{eq:p2} \hspp (Sequential, $O(1)$)\; 
	$\Delta = c_g v_5+c_y v_4+c_s v_6$ \hspp (Sequential, $O(1)$)\; 
	$s_k^{t+1} = c_g g_k + c_y y^t_k + c_s s^t_k$ \hspp (Parallel)\;
	$x_k^{t+1}=x_k^t+s_k^{t+1}$ \hspp (Parallel)\;
     }
  }
  \caption{Parallel Implementation}
  \label{alg:parimp}
\end{algorithm}

We coded the algorithm in C using OpenMP.  We have selected three test
problems from the CUTEst collection \cite{Cute:2004}, whose
dimensions can be varied to obtain small- to large-scale problems,
namely COSINE, NONCVXUN, and ROSENBR.  The objective functions of the
problems, and the initial points used in our tests are as follows:
\[
\begin{array}{c}
 f(x) = \sum_{i=1}^{n-1} \cos({-0.5*x_{i+1}+x_i^2}), \qquad x_i^0 =
 1.0,  \qquad \mbox{(COSINE)} \\
 f(x) = \sum_{i=1}^n x_i^2 + 4\cos(x_i), \qquad  x_i^0 = \log(1.0+i),
 \qquad \mbox{(NONCVXUN)} \\
 f(x) = \sum_{i=1}^{n-1} 100(x_{i+1}-x_i^2)^2+(1-x_i)^2, \qquad x_i^0 = 1.2, \qquad \mbox{(ROSENBR)}
\end{array}
\]
where $x_i$ denotes the $i$th component of point $x$ and $x^0$ is the
initial point. We note that the function and gradient evaluations of
all three test problems are in the order of $O(n)$. So, function
evaluations are not dominant to the computations of the new strategy.
Moreover, the evaluations of all three functions can be done in
parallel with one synchronization in computing the objective function
value.

To test the parallelization performance of the algorithm, we set the
dimension of the problems to $5$ million, $25$ million and $125$
million, and solve each of the resulting instances for various values
of the total number of threads, $p \in \{1,2,4,8,16\}$.  We run the
algorithm by setting $\eta=0.5$, and the step acceptability is checked
using condition \eqref{eq:accunc2} with $\rho=0.1$.  We set both the
maximum number of inner iterations and the maximum number of outer
iterations to $100$.  If the algorithm cannot reach a solution with
the desired accuracy
\[
\frac{\|\nabla f(x_k)\|}{\max\{||x_k||,1\}}<10^{-5}.
\] 
then we report the clock time elapsed until the maximum number of
outer iterations is reached. We summarize these results in
Table~\ref{tab:unc2test2} 
 We observe that the speed-up is close to linear for up to
8 cores. Since there is no data reuse, if we further increase the
number of cores, the main memory bandwidth becomes the
bottleneck. This is a common issue when memory-bounded computations,
like inner product evaluations, are parallelized.

 \begin{table}
 \centering
 \caption{Run times (in seconds) for varying values of $n$ and $p$}
 \label{tab:unc2test2}
 \begin{tabular}{lr|rrrrr}
 \hline\hline
 \multicolumn{1}{c}{Problem} & \multicolumn{1}{c|}{$n$} 
 & \multicolumn{1}{c}{$p=1$} & 
 \multicolumn{1}{c}{$p=2$} & \multicolumn{1}{c}{$p=4$} & \multicolumn{1}{c}{$p=8$} & \multicolumn{1}{c}{$p=16$} \\
 \hline
   & 5M 
   & 8.80 & 4.41 & 2.24 & 1.29 & 0.89 \\ 
   COSINE & 25M 
   & 43.98 & 22.07 & 11.19 & 6.33 & 4.35\\ 
   & 125M 
   & 219.77 & 111.74 & 57.98 & 31.62 & 21.47\\ 
 \hline
   & 5M 
   & 25.19 & 12.57 & 6.49 & 4.08 & 2.88 \\ 
  NONCVXUN & 25M 
  & 118.59 & 58.63 & 30.44 & 18.39 & 13.29 \\ 
   & 125M 
   & 496.45 & 249.43 & 127.38 & 75.43 & 55.07 \\ 
  \hline
    & 5M 
    & 45.93 & 22.69 & 11.89 & 7.91 & 6.11 \\ 
  ROSENBR & 25M 
  & 219.52 & 111.16 & 57.02 & 35.74 & 28.31 \\ 
    & 125M 
    & 1113.83 & 554.07 & 283.83 & 178.41 & 140.60 \\
   \hline
 \end{tabular}
 \end{table}

Finally, we consider the load balance issue. To observe the usage of
capacity and the workload distribution, we plot the CPU usage during
the solution of 1M-dimensional instance of the problem COSINE
(see Figure~\ref{fig:cpu}).  The plots are obtained by recording the
CPU usage information per second via the \texttt{mpstat} command-line
tool (see Linux manual pages). Figure~\ref{fig:cpu}(a) reveals the
idle resources when the program runs sequentially.  In fact, during
the sequential run, the average resource usage stays at a level of
$11.12\%$. This value is computed by taking the average usage of eight
cores. When eight threads are used, this average raises up to
$78.65\%$. Figure~\ref{fig:cpu}(c) shows the usage of all available
resources as well as the distribution of the workload.

\begin{figure}
\centering
\subfigure[$p=1$]{
\scalebox{0.38}{\includegraphics{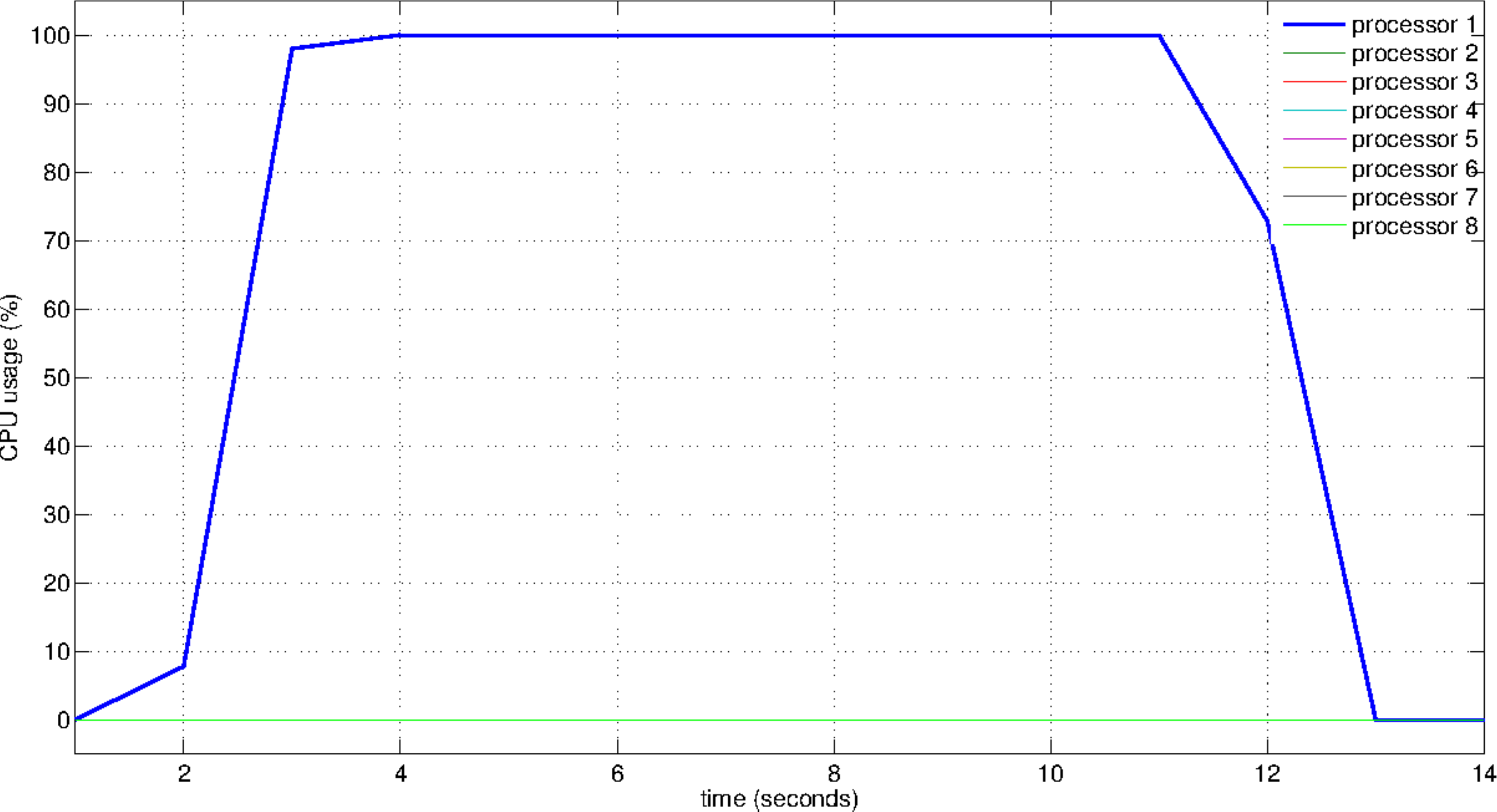}}
}
\subfigure[$p=2$]{
\scalebox{0.38}{\includegraphics{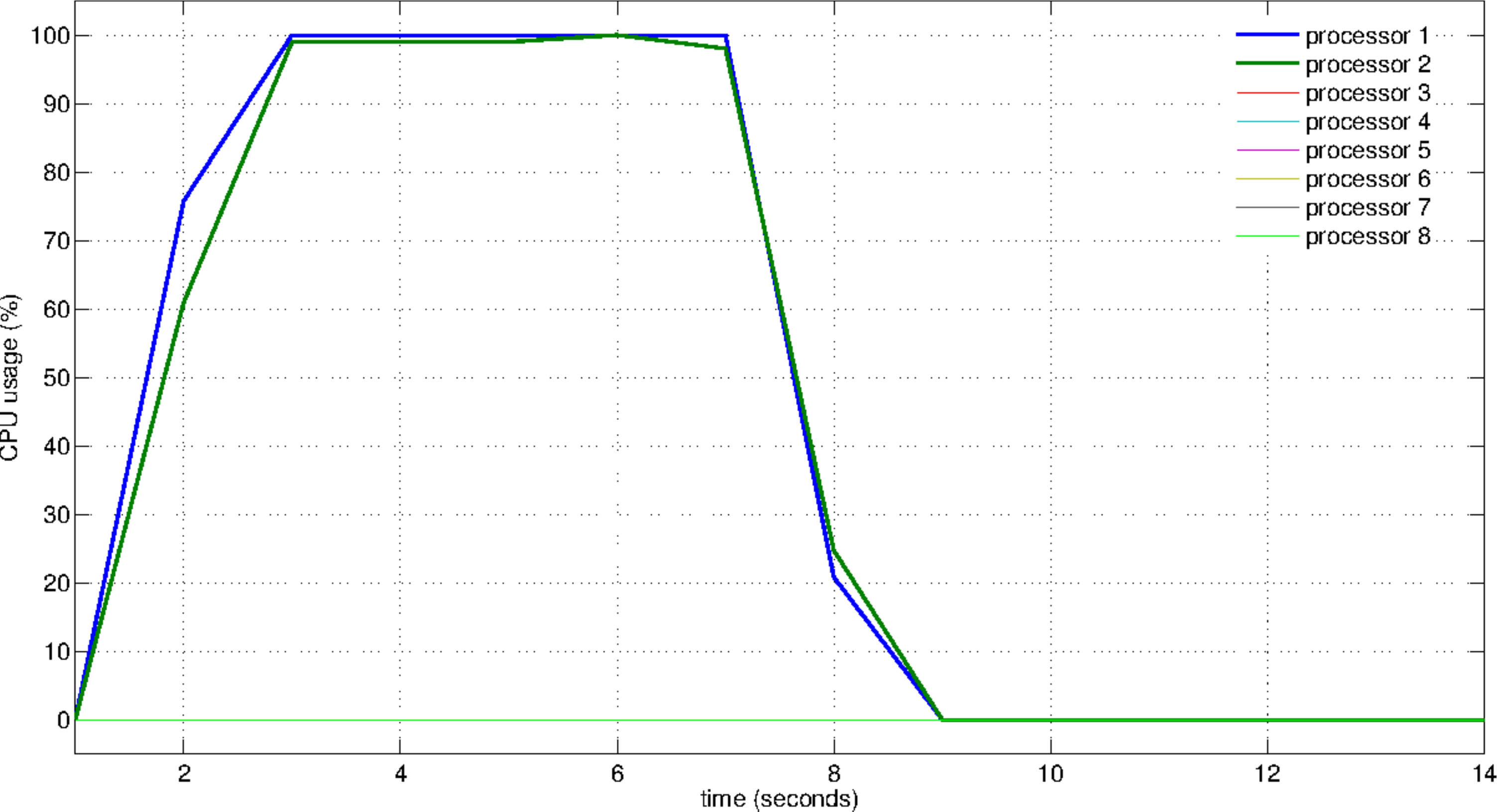}}
}
\subfigure[$p=8$]{
\scalebox{0.38}{\includegraphics{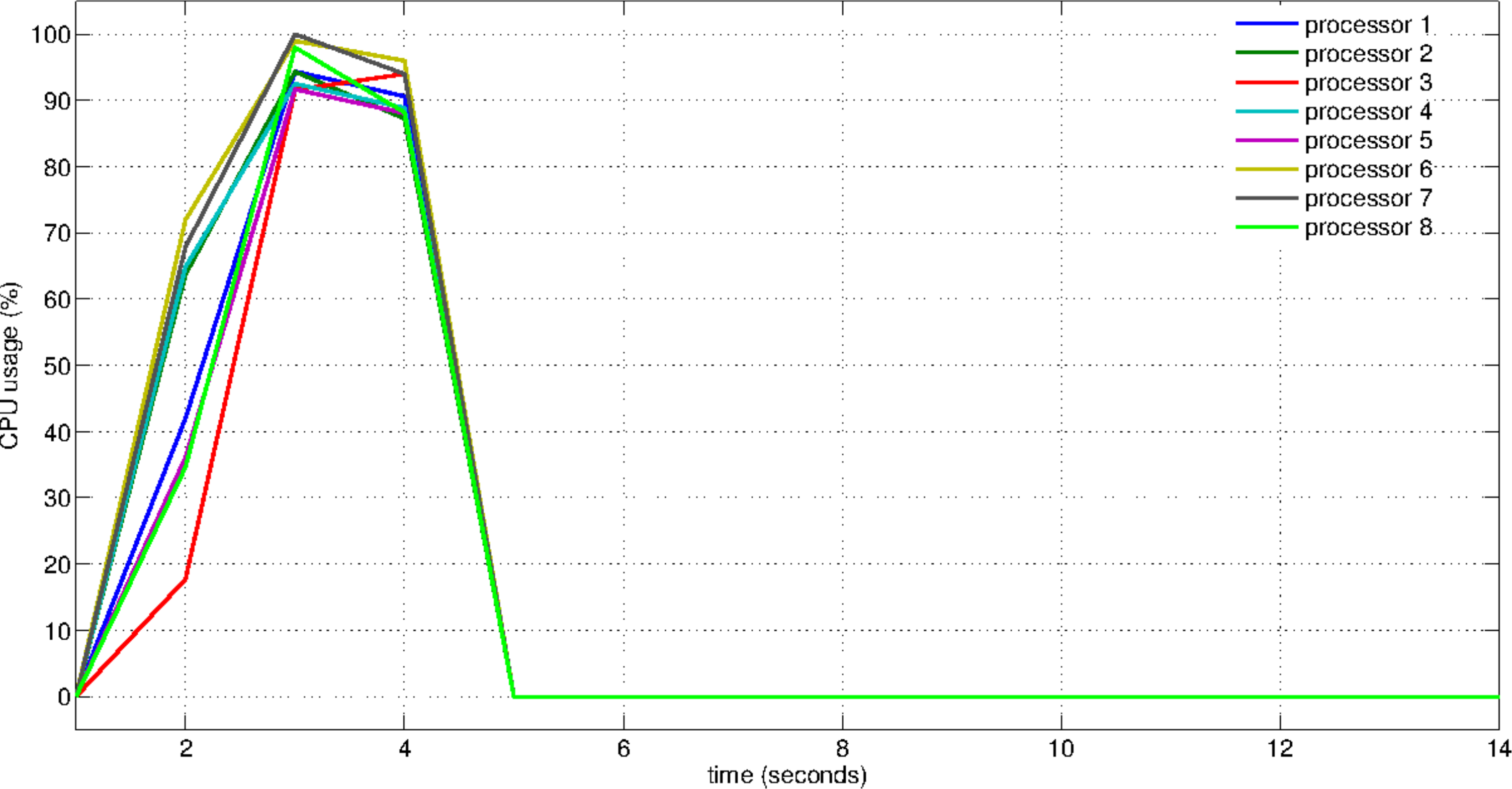}}
}
\caption{CPU usage (per second) during the solution processes with 1,2, and 8 threads}
\label{fig:cpu}
\end{figure}

\clearpage

\section{Conclusion and Future Research.} 
\label{sec:conc}



The current study is a part of our research efforts that aim at
harnessing parallel processing power for solving optimization
problems. Our main motivation here was to present the design process
that we went through to come up with an alternate globalization
strategy for unconstrained optimization that can be implemented on a
shared memory parallel environment.

The proposed algorithm works with a model function and considers its
optimization over an elliptical region. However, it does not employ a
conventional trust-region approach. Nor does it match with any one of
the well-known quasi-Newton updates. The basic idea is to use multiple
trial points for learning the function structure until a good point is
obtained.  These trials constitute the inner iterations. Fortunately,
all these computations of the algorithm are domain-separable with two
$O(1)$ synchronization points at each inner iteration.

Our numerical experience verifies that the direction updates of the
resulting algorithm can reduce the total number of function
evaluations required, and it is scalable to a degree allowed by the
inner iterations with a balanced distribution of the workload among
parallel processors.  Like any parallel optimization method, the
parallel performance of the proposed algorithm can be compromised if
function evaluations are computationally intensive and unfit for
parallelization.

We only solved a set of examples with our algorithm. It is yet to be
seen whether the algorithm is apt for solving other large-scale
problems, especially those ones arising in different applications.  In
our implementation, we chose the initial trial point simply by using
the negative gradient step. One may try to integrate other,
potentially more powerful but still parallelizable, steps to the
algorithm.  An example could be using a truncated Newton step; perhaps
computed by a few iterations of the conjugate gradient on the initial
quadratic model. Then switching to the globalization strategy, as
described here, could adjust the direction and the length of the
initial step.  Finally, our observations on the relationship of the 
new strategy with the quasi-Newton update formulas can be studied
further in future research.

\clearpage

\appendix

\section{Omitted Proofs.} 
\label{app:proofs}

\lemmanum{\ref{lem:tr}}{
If constraint \eqref{const} holds at inner iteration $t+1$ of
iteration $k$, then both $\alpha_k^0(s)$ and $\alpha_k^t(s)$ are
nonnegative, and they satisfy $\alpha_k^0(s) + \alpha^t(s)=1$.
}
\begin{proof}
  We have
  \[
    s\tr s_k^t - (s_k^t)\tr s_k^t  = (s-s_k^t)\tr (-(s-s_k^t) + s)  = -\|s - s_k^t\|^2 + s \tr s -
                                   s\tr s_k^t.
  \]
  Since constraint \eqref{const} implies $s\tr s \leq s\tr
  s_k^t$, we obtain 
  \[
  0 \leq s\tr s_k^t \leq (s_k^t)\tr s_k^t.
  \]
  The result follows from the definitions of $\alpha^0(s)$ and
  $\alpha^t(s)$.
\end{proof}

\lemmanum{\ref{lem:fact1}}{ Let $f$ be a convex function and consider
  the inner iteration $t+1$ of iteration $k$ with $\|s_k^t\| \neq
  0$. If $g_k\tr s_k^t \leq 0$, then $\|s_k^{t+1}\| \neq 0$ and
\[
\alpha^*:=\arg\min_{0 \leq \alpha \leq 1} m_k^t (\alpha s_k^t) = \min\left\{
\frac{(y_k^t)\tr s_k^t - f_k^t + f_k}{2(y_k^t)\tr s_k^t}, 1 \right\}.
\]
If we additionally have $f_k^t > f_k$, then $\alpha^* < 1$.
}
\begin{proof}
By using relation \eqref{eqn:cvxmodgra}, we have
\[
\nabla m_k^t(0) = g_k + \frac{1}{\|s_k^t\|^2}(f_k^t - (g_k^t)\tr s_k^t -
f_k)s_k^t.
\]
Since $f$ is a convex function, we know that $f_k^t - (g_k^t)\tr s_k^t
- f_k \leq 0$.  If $f_k^t - (g_k^t)\tr s_k^t - f_k = 0$, then
$\|\nabla m_k^t(0)\| = \|g_k\| \neq 0$ because $x_k$ is not a
stationary point.  Now, consider the case $f_k^t - (g_k^t)\tr s_k^t -
f_k < 0$, and suppose for contradiction that $\|\nabla m_k^t(0)\| =
0$.  Then,
\[
s_k^t = \frac{\|s_k^t\|^2}{f_k - f_k^t + (g_k^t)\tr s_k^t}g_k.
\]
Multiplying both sides with $g_k$, we obtain
\[
g_k\tr s_k^t = \frac{\|s_k^t\|^2}{f_k - f_k^t + (g_k^t)\tr s_k^t}\|g_k\|^2.
\]
Since $g_k\tr s_k^t \leq 0$, we obtain a contradiction and hence
$\|\nabla m_k^t(0)\| \neq 0$. This shows that $\|s_k^{t+1}\| \neq
0$. Note that
\[
(s^t_k)\tr \nabla m_k^t(0) = g_k\tr s_k^t + f_k^t - (g_k^t)\tr s_k^t -
f_k  = f_k^t -f_k - (s^t_k)\tr (g_k^t-g_k) \leq 0.
\]
This implies that $s_k^t$ is a descent direction for $m_k^t(s)$ at
$s=0$. Then, it is easy to solve the one-dimensional convex
optimization problem to obtain
\[
\alpha^* = \arg\min_{0 \leq \alpha \leq 1} m_k^t (\alpha s_k^t) = \min\left\{
\frac{(y_k^t)\tr s_k^t - f_k^t + f_k}{2(y_k^t)\tr s_k^t}, 1 \right\}.
\]
Again by using relation \eqref{eqn:cvxmodgra}, we have
\[
 \nabla m_k^t(s_k^t) = g_k^t - \frac{1}{\|s_k^t\|^2}(f_k + {g_k}\tr s_k^t - f_k^t)s_k^t,
\]
which gives 
\[
-(s^t_k)\tr \nabla m_k^t(s^t_k) = -(g_k^t)\tr s_k^t + f_k +g_k\tr s_k^t
- f_k^t = f_k-f_k^t - (s^t_k)\tr (g_k^t-g_k).
\]
If $f_k^t < f_k$ along with $g_k\tr s_k^t\leq 0$, then 
\[
0 < f_k^t -f_k \leq (s^t_k)\tr (g_k^t-g_k)
\]
holds. Therefore, we have $-(s^t_k)\tr \nabla m_k^t(s^t_k) \leq 0$,
which implies that $(-s_k^t)$ is a descent direction for $m_k^t(s)$ at
$s=s_k^t$. In this case, the minimizer along $s_k^t$ should be in the
interior. Thus, $\alpha^* < 1$.
\end{proof}

\end{document}